\documentclass[12]{article}

\usepackage{amsmath,amssymb,latexsym,amsthm}
\newtheorem{theorem}{Theorem}
\usepackage[bottom]{footmisc}
\newtheorem{lemma}[theorem]{Lemma}
\newtheorem{proposition}[theorem]{Proposition}
\newtheorem{definition}[theorem]{Definition}
\newtheorem{remark}[theorem]{Remark}
\newtheorem{example}[theorem]{Example}

\usepackage{booktabs}
\def\bn{\bigskip\noindent}
\def\R{{\mathbb R}}
\def\Q{{\mathbb Q}}

\usepackage{amsmath,amssymb,latexsym}

\title{The icosahedra of edge length 1}

\author{Karl-Heinz Brakhage\footnotemark[1],
Alice C. Niemeyer\footnotemark[2],
Wilhelm Plesken\footnotemark[2],\\
Daniel Robertz\footnotemark[3],
Ansgar Strzelczyk\footnotemark[2]}

\footnotetext[1]
{Institut f\"ur Geometrie und Praktische Mathematik, RWTH Aachen University, Templergraben 55, 52062 Aachen, Germany.
{\tt\small brakhage@igpm.rwth-aachen.de}}
\footnotetext[2]
{Lehrstuhl B f\"ur Mathematik, RWTH Aachen University, Pontdriesch \mbox{10-16}, 52062 Aachen, Germany.
{\tt\small Alice.Niemeyer@mathb.rwth-aachen.de},
{\tt\small Wilhelm.Plesken@mathb.rwth-aachen.de},
{\tt\small Ansgar.Strzelczyk@math.rwth-aachen.de}}
\footnotetext[3]
{School of Computing, Electronics and Mathematics, University of Plymouth, \mbox{2-5} Kirkby Place, Drake~Circus, Plymouth, Devon, PL4 8AA, United Kingdom.
{\tt\small daniel.robertz@plymouth.ac.uk}}

\begin{document}
\maketitle

  \begin{abstract}
  Retaining the combinatorial
  {\sc  Euclid}ean
structure  of  a  regular  icosahedron,  namely  the  20  equiangular
(planar) triangles, the  30 edges of length 1, and the 12  different vertices
together with the incidence structure, we 
investigate  variations of  the regular  icosahedron admitting
self-intersections of faces. We determine all rigid equivalence classes of
these icosahedra with non-trivial automorphism group and find one
curve of flexible icosahedra. Visualisations and explicit data for this paper
are available under  \\
{\tt http://algebra.data.rwth-aachen.de/Icosahedra/visualplusdata.html}.
  \end{abstract}


\section{Introduction} 
The regular icosahedron has already fascinated the ancient Greeks.  In
this paper  we investigate  variations of  the regular  icosahedron as
follows:  We  keep  the  combinatorial part  of  the  {\sc  Euclid}ean
structure  of  the  regular  icosahedron, namely  the  20  equiangular
(planar) triangles,  the 30 edges  of length  1, and the  12 different
vertices  with the  incidence  structure. We  drop  the assumption  of
convexity  and even  allow that  the triangles  penetrate each  other.
This results  into a system  of 30  quadratic equations over  the real
numbers for  $3\cdot 12$  indeterminates. The  real solutions  with 12
different vertices  we simply call  icosahedra. Having tried  to solve
these equations  for some time,  we have  come to the  conclusion that
solving them is a hard problem indeed.

\bn  We dedicate  this  paper  to the  late  Charles  Sims.  His  main
contribution to mathematics  was in group theory,  where, for example,
his  contributions  to  group  theoretic  algorithms  were  a  crucial
ingredient $\,$ in the $\,$ quest $\,$for$\,$ finite sporadic  \newpage 
\noindent simple groups.  Inspired
by his perseverance of tackling hard  problems we too have not given up
and, using group theory, have at  least come close to a classification
of all icosahedra allowing a non-trivial symmetry group.

\bn To be more specific, we may and do assume that the edge lengths of
the triangles  are all 1.  To get rid  of the translations  of 3-space
moving around  the icosahedra, we assume  that the vertices sum  up to
$0$, i.e.  the coordinate origin is  the center of mass  of the twelve
(equilibrated) vertices. Hence  an icosahedron is given  by a $3\times
12$-matrix $M$ whose  columns $V_1, \ldots, V_{12}$  give the standard
coordinates for the 12 vertices.  Having numbered the vertices in some
fixed way these icosahedra still allow the operation of the orthogonal
group of {\sc Euclid}ean 3-space: With $M$ also $OM$ is a solution for
every $O  \in O_3(\R )$.  Finally to get rid  of this group  action as
well, we pass over to the {\sc Gram}-matrix $G:=M^{tr} \cdot M$, which
is a $12  \times 12$-matrix with the following  obvious properties: It
is real, symmetric, positive semidefinite of rank at most 3. Its three
non-negative  eigenvalues  give some  idea  of  how the  vertices  are
distributed  around  their  center  of   mass.  Also  some  choice  of
eigenvectors  for the  three non-negative  eigenvalues yields  a normalized
choice for  the coordinate  matrix $M$ of  the icosahedron.  Since the
action of the orthogonal group has been factored out by the passage to
the {\sc Gram}-matrices, equivalence,  isometry, or isomorphism of the
latter is just conjugacy by permutation matrices.

\bn It is well known that  the combinatorial automorphism group $A$ of
the icosahedron  has order 120  and is isomorphic to  $C_2\times A_5$,
where the generator $d$ of the center $Z(A) = C_2$ is generated by the
permutation  interchanging combinatorially  opposite  vertices of  the
icosahedron. Here is a summary of our results.

\begin{theorem} \label{Uebersicht}
The subgroups  $U$ of  $A$ with  more than one  element that  arise as
symmetry group  of an  icosahedron fall into  11 conjugacy  classes of
subgroups  of $A$.  For each  $U$ we  list the  number of  equivalence
classes of icosahedra.
  \[
\begin{array}{|l|c|}
\hline
\mbox{Automorphism group} &\mbox{Number of}\\
U\leq A=C_2\times A_5 &\mbox{icosahedra}\\
\hline 
 C_2 \times A_5 & 2\\
 \hline 
 C_2 \times D_{10}& 4\\
 \hline 
 C_2 \times D_{6}& 2\\
 \hline 
 D_{10} \quad (\not\leq  A_5)& 3\\
 \hline 
 D_6 \quad  (\not\leq  A_5)& 2\\
 \hline 
 {C_2}^2 \quad (\ni d) & 1\\
 \hline 
 {C_2}^2 \quad (\not\ni d, \not\leq A_5) & 5\\
 \hline 
 {C_2}^2 \quad (\leq A_5) & 1\\
 \hline 
 C_2 \quad (\leq A_5)& 5\\
 \hline 
 C_2 \quad (\not\ni d, \not\leq A_5) & 10\\
 \hline 
 C_2 \quad (=\langle d \rangle)& \infty\\
 \hline
\end{array}
\]
\end{theorem}

\bn  For the  finite  cases we  shall produce  formulas  for the  {\sc
  Gram}-matrices,  called formal  {\sc Gram}-matrices,  most of  which
usually contributing more than one  real {\sc Gram}-matrix.  This will
be  discussed  in  Section~\ref{sec:formalGramMatrices}.   Further  it
turns out, most cases of subgroups  $U$ split up into several subcases
distinguished by the linear action of $U$ on $3$-space. This splitting
is an essential step for  this classification. Details are discussed in
Section~\ref{SymLin},  cf.  also  the final  Table~\ref{endgue}.   The
final  case   with  infinitely  many  solutions   is  investigated  in
Section~\ref{CurvIco}.  We only  present  existence  proofs with  some
numerical  approximations. In particular we have a  simulation of  a
curve of deformable icosahedra, which is obtained by solving a
certain ordinary differential equation numerically.

The problem addressed in this paper can also be viewed as the
problem of classifying embeddings into Euclidean 3-space of the graph
that is defined by the incidence structure of the icosahedron, with
prescribed edge lengths and symmetry. The question whether continuous
families of embeddings exist is known as the question whether the considered
framework is rigid. A common approach to investigate rigidity is to check
infinitesimal rigidity and conclude in the affirmative case that the framework
is rigid for sufficiently generic embeddings. In \cite{Schulze} this approach was
adapted to frameworks with symmetry. Infinitesimal rigidity implies rigidity also
for sufficiently generic embeddings with specified symmetry.
The method of \cite{Schulze} would be an alternative to show the existence of the
curve of deformable icosahedra with symmetry, but the question whether prescribing
the edge lengths is compatible with the genericity assumption would still have to
be addressed.

As  a  consequence  of   the  easier  parts  of  our
computations we mention:

\begin{proposition}\label{einlgeomd-1}
There are exactly four isometry classes of  icosahedra $($with all edge lengths  $1)$
such that the midpoints of combinatorially opposite vertices have the center of mass of the 
icosahedron as their midpoint. 
\end{proposition}

\section{Symmetry and linear action on 3-space}\label{SymLin}
The combinatorial automorphism group $A\cong C_2 \times A_5$ of the abstract icosahedron  can be generated by the permutations 
\begin{align*} 
a & = (1,2)(3,4)(5,7)(6,8)(9,11)(10,12),\\
b & = (1,10)(3,9)(2,12)(4,11)(5,6)(7,8), \\
c & = (1,7)(2,3)(4,11)(5,12)(6,8)(9,10),\\
d & = (1,12)(3,9)(2,10)(4,11)(5,7)(6,8).
\end{align*}
In  particular  $d$  generates   the  center  $C_2$  and  interchanges
combinatorially opposite vertices, whereas  the first three generators
$a,b,c$ form a  minimal generating set of involutions  for $A_5$.  The
triangles are obtained  as the orbit of $\{ 1,2,3\}$,  the 30 edges as
orbit of $\{  1,2\}$, the 30 diagonals of combinatorial  distance 2 as
the orbit of $\{ 3,4\}$, and  finally the 6 diagonals of combinatorial
distance 3 as  the orbit of $\{  1,12\}$, which one also  finds in the
2-cycles  of the  generator  $d$ of  the  center of  $A$.  There is  a
geometric  bijection called  ``orthogonal diagonal''  between the  two
orbits $\{ 1,2\}^A$ and $\{ 3,4\}^A$ as follows
\[
 \omega : \{ 1,2\}^A \to  \{ 3,4\}^A: \{ i,j\} \mapsto \{ s,t\} \mbox{ iff } \{ i,j,s \} , \{i,j,t \} \in \{1,2,3\}^A.
\]
Finally let
$\pi : A \to {\rm GL}(12, \R)$ denote the natural representation of $A$ by permutation matrices.

\begin{definition}
The combinatorial automorphism group $A$ acts on the set ${\cal G}$
of all {\sc Gram}-matrices of icosahedra by
$$A \times {\cal G}: (g,
G) \mapsto \pi (g)^{tr}\cdot G\cdot\pi (g)= (G_{ig,jg})_{1\leq i,j\leq
  12}$$
and the  stabilizer of $G \in  {\cal G}$ is the  automorphism group of
$G$ or  the associated icosahedron. Note,  $\pi (g)^{tr}=\pi (g)^{-1}$
so that the trace is an invariant for this action.
\end{definition}

\bn
We now demonstrate 
a method splitting up the case of an automorphism group $U$ into more manageable
subcases by using the linear action of $U$ on 3-space. To this end
we have the following elementary lemma from linear algebra:

\begin{lemma}
Let $G\in \R^{n \times n}$ be a symmetric, positive semidefinite matrix of rank $k$.\\
1.) There exists a real $k\times n$-matrix $M$ with $M^{tr}M=G$.\\
2.) $G$ and $M$ have the same row space.\\
3.) If $L \in \R^{k\times n}$ also satisfies  $L^{tr}L=G$, then there is
a unique matrix $g \in \R^{k \times k}$ with $L=gM$. Moreover $g$ is orthogonal.\\
\end{lemma}

\begin{proof}
 1.) Let $(E_1, \ldots, E_{n})$ be an orthonormal basis of eigenrows of $G$ with eigenvalues
 $\lambda_1>0,\ldots, \lambda_k>0, \lambda_{k+1}= \ldots =\lambda_{n}=0$. Then $M$ can be chosen
 with the rows $\frac{1}{\sqrt{\lambda_i}} E_i$ for $i=1,\ldots,k$,
 because $E_i^{tr}E_i$ represents the
 orthogonal projection of $\R^{1\times n}$ onto $\langle E_i\rangle$.\\
 2.) Obvious. 3.) Clear from 2.).
\end{proof}

\bn
Applying this to our situation we get the following:

\begin{lemma}
 Let $G\in {\cal G}$ be the {\sc Gram}-matrix of an icosahedron and
 $U\leq A$ its automorphism group.  Then there is a faithful
 orthogonal representation $\delta: U \to O_3(\R)$ and a matrix $M \in
 \R^{3 \times 12}$ called coordinate matrix such that 
 \[
  \delta (g)\cdot M = M\cdot \pi (g)  \mbox{ for all } g \in U \mbox{ and } G:=M^{tr}\cdot M
 \]
\end{lemma}

\bn The name coordinate matrix is chosen because its columns represent
the  vertices of  the icosahedron  in  3-space. Since  this matrix  is
unique only up  to orthogonal transformations, the  field generated by
its entries is not canonical  unlike the corresponding field $F_G$ for
the {\sc Gram}-matrix $G$. But from  a computational point of view the
coordinate matrices are more  accessible than the {\sc Gram}-matrices.
We therefore go for them as follows: For each minimal subgroup $U$ (up
to conjugacy)  of $A$  and each  faithful orthogonal  representation $
\delta : U \to O_3(\R )$  compute the relevant coordinate matrices $M$
among  the  intertwining  matrices  of $\delta$  and  the  restriction
$\pi_{|U}$. Find  a set  of representatives of  the $A$-orbits  of the
{\sc Gram}-matrices $G:=M^{tr}\cdot M$. Most of these calculations can
be done with the formal matrices, as we shall see in the next section.

\section{Formal {\sc Gram}-matrices}
\label{sec:formalGramMatrices}

Assume $\{1\}  \neq U \leq A$. Let $e_1, \ldots ,e_{12}$ be the standard basis of $\R^{12 \times 1}$.
In order to determine the $U$-invariant {\sc Gram}-matrices $G$,
we define an ideal $I$ generated by  
$(e_i-e_j)^{tr} G(e_i-e_j)-1$ for all $\{i, j\} \in \{ 1, 2 \}^A$ and the  $4\times 4$-minors
of $G$, where the entries $G_{i,j}$ of $G$ are equated to variables $y_1, \ldots, y_n$ 
in correspondence with the $U$-orbits on 
the positions in the symmetric $12 \times 12$-matrix $G$. Then $I$ is an ideal in  
 $R:= \Q [y_1, \ldots, y_n]$. It will turn out in all but one case, $R/I$ is finite dimensional over $\Q$
 so that we concentrate on maximal associated primes first. Here is a complete list of all necessary 
conditions for the case of zero dimensional $I$.

\begin{definition}\label{formalGram}
A maximal ideal $m \trianglelefteq R$ associated to  $I$ (in the primary decomposition) 
is called \underline{relevant} if the 
following three conditions are satisfied:\\
1.) The rank of the image matrix $(G_{i,j}+m)\in (R/m)^{12 \times 12}$ is at most 3;\\
2.) $U=\{ g \in A \mid \pi(g)^{tr} (G_{i,j}+m)\pi(g)=(G_{i,j}+m) \}$;\\
3.) there exists a \underline{relevant} real embedding $\iota : R/m \to \R$, i.e.  $(\iota (G_{i,j}+m))$  is positive semidefinite. \\
$G(m):=(G_{i,j}+m)$ is called the associated \underline{formal {\sc Gram}-matrix}, 
the residue class  field $R/m$ (as well as its abstract isomorphism
type) is called the \underline{field of} \underline{definition} $F_{G(m)}$ of $G(m)$.
The \underline{field degree}  $[R/m:\Q]$ is denoted by $d_{G(m)}$, the
number of real embeddings of $R/m$ by $r_{1,G(m)}$, the number of
relevant real embeddings of $R/m$ by $r_{G(m)}$, and finally
$r_{f,G(m)}$ denotes  the number of relevant real embeddings with
$(\iota (G_{i,j}+m))$ pairwise inequivalent. We  refer to $r_{G(m)}$
as \underline{generosity} and to $r_{f,G(m)}$ as
\underline{contribution} of $G(m)$. 
\end{definition}

\bn
Given a real {\sc Gram}-matrix representing an icosahedron whose entries are algebraic numbers, then of course 
the field generated by all its entries is finite over $\Q$, has a primitive element, and at least one real embedding. So one might hope, that it defines a formal {\sc Gram}-matrix. This, however, need not be 
the case. One can compute its automorphism group and relate it to the ideal $I$ above. Though one finds
a maximal ideal yielding this {\sc Gram}-matrix, it might be so that the maximal ideal is embedded
rather than associated to $I$. Such cases exist, because later on, we shall 
exhibit an example of a one dimensional equational ideal $I$, indeed we find
a curve of icosahedra.

With this definition the reader can already interpret most of the final Table~\ref{endgue}. Here are
some further comments:

\begin{remark}
1.) There are only finitely many  formal  {\sc Gram}-matrices up to equivalence.\\
2.) For each formal {\sc Gram}-matrix $G(m)$ the normalizer $ N_A(U)$ permutes
the maximal ideals above. The stabilizer $D(G(m))$ (called
decomposition group)  of $m$ in $ N_A(U)$ permutes the relevant embeddings
of $R/m$ and $D(G(m))/U$ acts 
faithfully on $R/m$ by field automorphisms. \\
3.) In particular the length of the orbit of $G(m)$ under $ N_A(U)$ is of the form
$[ N_A(U):D(G(m))]\cdot [D(G(m)):U]$, where the first factor is the length of the 
$N_A(U)$-orbit of $m$. 
\end{remark}

It is convenient for checking the relevance of real embeddings as well
as for  tabulating formal  {\sc Gram}-matrices  to choose  a primitive
element of $R/m$.  In case the trace of $G(m)$  generates $R/m$, it is
an ideal  candidate because  it gives the  formal {\sc  Gram}-matrix a
canonical form.   The last column in  Table~\ref{endgue} indicates the
minimal polynomial  of the  trace of the {\sc  Gram}-matrix over  $\Q$. In
particular  one can  read  off  the degree  of  $R/m$  over the  field
generated  by the  trace.  In seven  cases the  trace  is a  primitive
element.  Five of  these cases,  namely the  simplest ones,  where the
field degree  is two, can  be explained  easily in geometric  terms by
starting  from the  regular icosahedron  from ancient  Greece. In  all
these  cases $R/m$  is  isomorphic to  $\Q[x]/(x^2-5)$  and both  real
embeddings are relevant. The  following construction will explain this
phenomenon:

 \begin{remark}
If  the 5  neighbour vertices  of a  vertex $V$  lie in  a plane,  one
obtains  a  new  icosahedron  from  the  given  one  by  applying  the
orthogonal  reflection  with  the  fixed plane  spanned  by  the  five
neighbour vertices of  $V$ and the five triangles adjacent  to $V$ and
keeping the  rest of the  icosahedron.  Though this  operation changes
the center of mass,  it can be adjusted by a  translation and does not
change  the field  of  definition,  nor $r$  or  $r_f$.  We call  this
operation \underline{denting} at $V$.
 \end{remark}

 So the  regular icosahedron  with automorphism group  $C_2\times A_5$
 yields the  one with  automorphism group $D_{10}$  by denting  at one
 vertex $V$. Adding another dent at combinatorial distance 2 away from
 $V$  yields  the  one  with automorphism  group  ${C_2}^2$  and  with
 distance  3  away  the  one   with  automorphism  group  $C_2  \times
 D_{10}$.  Finally there  is  the  one with  three  dents of  pairwise
 distance $2$ with automorphism group $D_6$.

\begin{table}
  \begin{tabular}{@{}llrrrrl@{}}
\toprule
$S\!:=\!{\rm Stab}_A$ & $ {\rm Syl}_2(S)$ & $ d_G$ & $r_{1,G} $ & $ r_G $ & $r_{f,G}$ & $\mbox{Trace relation}$\\
\midrule
$    {C_2}^2$ & $ \langle a,d\rangle \frac{\bullet -}{+}
    $ & $8$ & $4 $ & $2$ & $1 $ & $ \lambda^4-\frac{76}{3}\lambda^3+238\lambda^2-\frac{4964}{5} \lambda+\frac{23767}{15}$\\
\specialrule{1pt}{2pt}{3pt}
$ C_2 \times A_5$ & $\langle a,b,d\rangle $ & $2$ & $ 2 $ & $2$ & $2  $ & $\lambda^2-15\lambda+45$\\
 \midrule
$ C_2 \times D_{10}$ & $\langle a,d\rangle  \frac{-\bullet}{+}$ & $2$ & $ 2 $ & $2$ & $2  $ & $\lambda^2-15\lambda+\frac{269}{5}$\\
\specialrule{1pt}{2pt}{3pt}
$ {C_2}^2   $ & $\langle a,bd\rangle  \frac{- +}{+}$ & $ 2$ & $ 2 $ & $ 2$ & $ 2 $ & $\lambda^2-\frac{71}{5}\lambda+\frac{10561}{225}$\\
   \midrule
$   C_2 \times D_{10}$ & $\langle a,d\rangle  \frac{-+}{+} $ & $4$ & $ 2 $ & $2$ & $2 $ & $\lambda^4-18\lambda^3+\frac{583}{5}\lambda^2-\frac{1658}{5}\lambda +\frac{9101}{25}$\\
    \midrule
$C_2 \times D_{6}$ & $\langle a,d\rangle \frac{-+}{+} $ & $4$ & $ 4 $ & $2$ & $2 $ & $\lambda^4-26\lambda^3+243\lambda^2-970\lambda +1397$\\
  \midrule
${C_2}^2 $ & $\langle a, bd \rangle\frac{-+}{+} $ & $24$ & $ 10 $ & $6$ & $3 $ & $ \lambda^{12}-\frac{5179\cdot 2^2}{3^2\cdot 5^2}\lambda^{11}\pm \cdots$\\
   \midrule
${C_2}^2$ & $\langle a,b \rangle \frac{--}{-}$ & $30$ & $ 18 $ & $6$ & $1 $ & $\lambda^5-\frac{117}{2}\lambda^4\pm \cdots $\\
  \midrule
$C_2$ & $\langle a \rangle\quad - $ & $172$ & $ 48 $ & $20$ & $5 $ & $\lambda^{43}-\frac{73\cdot 7\cdot 11\cdot 461687}{2^2\cdot3^3\cdot5^2\cdot 29\cdot 79} \lambda^{42}\pm \cdots $\\
\specialrule{1pt}{3pt}{3pt}
$D_{10}$ & $\langle ad\rangle\quad +$ & $2$ & $ 2 $ & $2$ & $2 $ & ${\lambda}^{2}-{\frac {44}{3}}\,\lambda+{\frac {2131}{45}}$\\
 \midrule
$D_{6}$ & $\langle ad\rangle\quad +$ & $2$ & $ 2 $ & $2$ & $2 $ & $ {\lambda}^{2}-{\frac {68}{5}}\,\lambda+{\frac {1111}{25}}$\\
   \midrule
$ C_2$ & $ \langle ad\rangle \quad +$ & $36$ & $ 12 $ & $8$ & $4 $ & $ \lambda^{18}-\frac{1106}{9}\lambda^{17}\pm \cdots$\\
\midrule
$ C_2$ & $\langle ad \rangle\quad + $ & $168$ & $ 40$ & $24$ & $6 $ & $\lambda^{42}-\frac{ 2 \cdot 719 \cdot 1223 }{
 3^{3} \cdot 5 \cdot 43} \lambda^{41}\pm \cdots $\\
\midrule
 $ D_{10}$ & $\langle ad \rangle \quad +$ & $4$ & $ 2$ & $2$ & $1 $ & $\lambda^2-\frac{26}{3}\lambda+\frac{149}{9}$\\ \hline 
\end{tabular}
\caption{List of all formal {\sc Gram}-matrices with symmetry with $a,b,c,d$ as defined in Section \ref{SymLin}}\label{endgue}
\end{table}

 \begin{theorem}\label{Hauptsatz}
There are up to equivalence 14 formal {\sc Gram}-matrices with non-central automorphism group resulting into 35 isomorphism classes
of real {\sc Gram}-matrices as specified in Table~\ref{endgue}.
\end{theorem}
 The final item in this table to be explained is the strange
 symbol in the second  column
for proper subgroups that are direct products.
 It shows two symbols above and one symbol below a horizontal line. The two symbols on top indicate the trace values of the two generators of the group, and the one below 
of their product. Here $+,-,\bullet$ stands for $1, -1, -3$ resp.. These symbols are also used for the
trace of the generator of $C_2$ in the last 6 cases.

A final comment on the results: The denominators of the coefficients in the trace relations involve only very few
prime factors, which already show up in the truncated version in the last column, e.g. for $d_G=172$
the primes are $2, 3, 5, 29, 79$, whereas the regular icosahedron and $C_2 \times D_6$-case have
an empty set of primes for the denominators. 

\bigskip
\noindent
We  finish  these introductory  sections  with  some comments  on  the
computations. In our first approach  we used the system {\sc Bertini},
cf. \cite{Ber}, which  was very helpful indeed, since  it created some
idea of  an answer. Due  to the complexity of  the problem it  did not
produce a  completeness proof. Later we  found out that it  missed one
isomorphism class, which is distinguished by the existence of coplanar
neigbouring triangles.  Further, {\sc Bertini} gives decimal expansions
of the  solution which for our  purposes still have to  be turned into
algebraic  numbers. The  number theoretic  methods available  for this
often failed for  field degrees above 40. The degrees  we found by the
methods of this paper were  sometimes considerably larger. Also at the
time we had no upper bound for these degrees.  Moreover we had no good
way to reduce the number of solutions by a priori splitting into cases
guided  by  symmetry.  So  we  had  to  deal  with  more  than  $10^4$
solutions.  But   on  the   positive  side   {\sc  Bertini}   gave  us
approximations of  two solutions  with trivial symmetry,  which easily
can be turned into an existence proof. Our present approach, presented
in  this  paper,  relies  on   formal  computations  using  the  MAGMA
\cite{Magma}  {\sc Groebner}-basis  functionality  and the  involutive
basis functionality  of GINV,  cf. \cite{GINV,BCG03}.  Anyone  interested in
the actual formal {\sc Gram}-matrices  can find them in MAPLE-readable
form on our homepage \cite{Home}  for this paper. Anyone interested in
the actual computations  leading to a proof  can find MAPLE-worksheets
also  on  this homepage.   Possibly  we  shall also  produce  detailed
comments  on  the  properties  of   the  various  solutions  on  these
pages. 

\section{Classifying the formal {\sc Gram}-matrices}\label{Classify}
Instead  of   working  with   the  definition   of  the   formal  {\sc
  Gram}-matrices  in  Definition~\ref{formalGram},   we  go  via  the
coordinate  matrices   $M$  giving  us  the   {\sc  Gram}-matrices  as
$M^{tr}M$.  We get  rid  of  the operation  of  the three  dimensional
orthogonal  group  on  the   coordinate  matrices  by  fixing  certain
entries. Also we do  not go through all subgroups of  $A$, but look at
the  minimal subgroups  and compute  the full  automorphism groups  of
their fixed {\sc Gram}-matrices. Note,  the automorphism groups of the
formal {\sc Gram}-matrices are equal to ones obtained by relevant real
embeddings, and are easily obtained by a stabilizer computation.

\bn
As for the minimal subgroups, they are generated by an element of prime order. Here is a complete list up to conjugacy
with  $a,b,c,d$ as defined at the beginning of Section \ref{SymLin}:
\[
\langle abc\rangle \cong C_3, \quad
\langle ac\rangle \cong C_5, \quad
\langle a\rangle \cong C_2, \quad
\langle d\rangle \cong C_2, \quad
\langle ad\rangle \cong C_2. 
\]

\bn
By the next lemma, the first two cases can be discarded. 

\begin{lemma}
If the {\sc Gram}-matrix of an icosahedron is fixed by an element of order $3$ or $5$, then its 
automorphism group also has an element of order $2$.
\end{lemma}

\begin{proof}
1.) As for the case of order 3 we may assume that the {\sc
  Gram}-matrix is fixed by $abc$. There is  (up to conjugacy under the
orthogonal group)  just one faithful orthogonal representation
$\delta $ of $U:=\langle abc\rangle \cong C_3$ of degree 3.
$\delta$ is the orthogonal sum of $\delta_1: abc\mapsto 1$ and an
irreducible representation $\delta_2$ of degree 2,
i.e. $\delta=\delta_1\bigoplus \delta_2$. Character theory tells us
that the first row of the searched coordinate matrix $M$ lies in a
space of dimension $4$, and the remaining  submatrix consisting of the
last two rows lies in a space of dimension $4\cdot 2$, so that we have
$11$ indeterminates for our equations to be solved. But we can do
better: The 4-dimensional space for the first row reduces to a
3-dimensional space because the sum of the entries is 0 (center of
mass condition). The 8-dimensional space for the remaining two rows
reduces to a 7-dimensional space  since the centralizer of $\delta_2$
in ${\rm O}_2(\R )$ can be applied to make one entry in the, say, first
row $0$. Now comes the computation of the solutions for the matrices
thus composed, which  amounts to solving a system of $30$ (= number of
edges of the icosahedron) quadratic equations  in 9
indeterminates. The involutive basis gives us a residue class algebra
of dimension 128 over the  ground field $\Q [\sqrt{3}]$ which has
exactly 42 residue class fields, of which 18 yield a formal {\sc
  Gram}-matrix with 12 different rows. Computing the orbits under the
action of $A$ yields orbit  length $1$ (four times), $20$ (ten times),
and $10$ (four times), implying that each stabilizer has order 6.
Note this calculation was purely on the formal level without computing
real embeddings.\\ 
2.) As for the case of order 5 we may assume that the {\sc
  Gram}-matrix is fixed by $ac$. There is  up to conjugacy under the
orthogonal group and algebraic conjugacy just one irreducible
representation $\delta_2$ of degree 2 expressed in the roots of the
$\Q$-irreducible polynomial $\lambda^4-(5/4)\lambda^2+5/16$, which has
both $\sin(2\pi/5)$ and $\sin(4\pi/5)$ amongst its roots. We again
have $\delta=\delta_1\bigoplus \delta_2$, obtain a vector space of
dimension $3+3$ containing the (formal) coordinate matrix $R$. The
quadratic equations lead to 24 solutions. This time all solutions
yield 12 different columns for the {\sc Gram}-matrix. The lengths of
the orbits under $A$ are $1$ (four times), $6$ (eight times), and $12$
(twelve times). Hence again all stabilizers have an order  
divisible by 2.
\end{proof}

\bn Hence we are left with the  task of finding and analyzing the {\sc
  Gram}-matrices  fixed  by  one  of  the  three  groups  of  order  2
above.  Each  one of  them  has  (up  to equivalence)  three  faithful
orthogonal representations of degree 3: the generator might get mapped
onto one of
\[
 {\rm diag}(1,-1,-1), \quad
 {\rm diag}(-1,1,1), \quad
 {\rm diag}(-1,-1,-1), 
\]
so  that  we  are  left  with  $3  \cdot  3=9$  cases,  summarized  in
Table~\ref{cases}. The first  two numbers in each box  give the number
of  isomorphism classes  of formal  and of  real {\sc  Gram}-matrices,
which come out  in the end for that particular  case. This is followed
by a  reference to the corresponding  lemma.  Of course the  cases are
not disjoint because many {\sc Gram}-matrices have automorphism groups
of orders bigger than 2.

\begin{table}
\[
 \begin{array}{|c|c|c|c|}
  \hline
  &  {\rm diag}(1,-1,-1) & {\rm diag}(-1,1,1)&  {\rm diag}(-1,-1,-1)\\
  \hline
  a & 8, 19 \quad \mbox{  cf. } \ref{1,1} & 0, 0\quad \mbox{  cf. } \ref{1,2} & 1, 1 \quad \mbox{  cf. } \ref{1,3}\\
  \hline
  d & ? \quad \mbox{  cf. } \ref{2,1}&  2,4 ,\quad  \mbox{ cf. } \ref{2,2}& 2, 4\quad  \mbox{ cf. } \ref{2,3}\\
  \hline
  ad & 0, 0 \quad \mbox{  cf. } \ref{3,1}&12,29 \quad  \mbox{  cf. }\ref{3,2} & 0, 0\quad  \mbox{   cf. } \ref{3,3}\\
  \hline
 \end{array}
\]
\caption{Cases distinguished by linear action }\label{cases}
\end{table}

\bigskip \noindent
One of these cases is ruled out immediately:

\begin{lemma}\label{3,3}
 There is no solution in the case of $\delta: \langle ad \rangle \to {\rm O}_3( \R ): ad \mapsto {\rm diag}(-1,-1,-1)$.
\end{lemma}

\begin{proof}
 The permutation $ad$ has fixed points, whereas $ {\rm diag}(-1,-1,-1)$ has no nonzero fixed vectors in $\R^3$.
\end{proof}

\bn
Slightly more complicated, but still very easy are the next two cases:\\

\begin{lemma}\label{3,1}
 There is no solution in the case of $\delta: \langle ad \rangle \to {\rm O}_3( \R ): ad \mapsto {\rm diag}(1,-1,-1)$.
\end{lemma}

\begin{proof}
To  set up  the coordinate  matrix $M  \in \R^{3  \times 12}$  and the
equations for  its entries is done  in the obvious way:  The first row
must be fixed  under the permutation $ad$ and the  entries must add up
to 0. Since $ad$ has 4  fixed points, this leaves $8-1$ indeterminates
for the first row.  The second and third row lie  in the eigenspace of
the permutation matrix  for $ad$ for the eigenvalue $-1$,  which is of
dimension 4,  so that we  end up  with $7+4+4$ indeterminates  and $6$
zero entries in $M$. We could still get rid of one more entry, but now
already there is  no solution, as a short  calculation with Involutive
shows.  
\end{proof}

\begin{lemma}\label{1,2}
 There is no solution in the case of $\delta: \langle a \rangle \to {\rm O}_3( \R ): a \mapsto {\rm diag}(-1,1,1)$.
\end{lemma}

\begin{proof}
Computing the eigenspaces for the  permutation matrix for $a$ leads to
the 6-dimensional  vector space  for the first  row of  the coordinate
matrix $M$, and a common 6-dimensional  space for the second and third
row.  This second  space  can  be reduced  to  a 5-dimensional  space,
because the sum of the columns is zero. Hence we get 16 indeterminates
occurring in  $M$. We can reduce  by 1, because the  first two columns
have  distance  1,  which  yields   one  quadratic  equation  for  one
indeterminate. Plugging that solution in yields solutions for $M$, but
they   produce   only   three   different   columns   for   the   {\sc
  Gram}-matrix. Hence there is no  solution with 12 different vertices
in this case.  
\end{proof}

\noindent
Concerning the results of the computations in case there are
solutions, we note the following.  

\begin{remark}
The icosahedra invariant under some element  $x \in A$ are permuted by
the  action of  the  normalizer  of $\langle  x\rangle  $  in $A$.  In
particular, the  ones invariant  under $d$ come  in orbits  under $A$.
The ones  fixed by $a$ or  $ad$ distribute themselves in  orbits under
the centralizer of $a$ in $A$,  which is the {\sc Sylow} 2-subgroup of
$A$ containing  $a$. In  particular the lengths  of the  latter orbits
divide 4, and the former ones 60.
\end{remark}

\begin{lemma}\label{1,3}  In case $\delta (a) = {\rm diag}(-1,-1,-1)$ there
  is one formal {\sc Gram}-matrix and one
real  {\sc Gram}-matrix, with details listed in the following table:
\[
\begin{array}{|c|c|c|c|c|}
\hline
\mbox{Stabilizer in } A&  d_G, r_{1,G} & r_G & r_{f,G}& \mbox{cf. also La.} \\
 \hline
 \langle a,d\rangle\cong {C_2}^2&8,\quad 4 &2&1 & \mbox{\ref{2,1}, \ref{3,2}}\\
 \hline 
\end{array}
\]
The field of definition $F_G$ has four real embeddings. The {\sc Galois} group of the normal 
closure of $F_G$ is the symmetric group $S_8$. The element $b \in A$ interchanges the two relevant 
real embeddings of the formal {\sc Gram}-matrix.\\
The linear action of the stabilizer can be chosen as $\delta (d) = {\rm diag}(1,-1,-1)$,   $\delta (ad) = {\rm diag}(-1,1,1)$.
\end{lemma}

\begin{proof}
The three rows  of the coordinate matrix $M \in  \R^{3 \times 12}$ lie
in the  eigenspace of $\pi  (a)$ for the  eigenvalue $-1$, so  that we
have 18 indeterminates. The center  of mass condition is automatically
fulfilled. Because of the action of the 3-dimensional orthogonal group
we   may   assume   that   the   first   column   is   of   the   form
$(*,0,0)^{tr}$. Since $a$  interchanges 1 and 2, the  second column is
of the same  shape. Since the first  three columns form a  face of the
icosahedron, the  stabilizer of the  first column in ${\rm  O}_3(\R )$
can be used to force the third column to be of shape $(0,*,0)^{tr}$ so
that  the first  three  columns are  essentially uniquely  determined.
After this the  relations generate a maximal ideal  with residue class
field  of  degree   16  over  the  rationals.    The  associated  {\sc
  Gram}-matrix is then  defined over a field of degree  8, which has 4
real embeddings.  But for  two of them  the {\sc  Gram}-matrix becomes
indefinite. The automorphism group for the formal {\sc Gram}-matrix is
$\langle a,d\rangle  \leq A$, in  particular isomorphic to  $V_4$. But
the centralizer of  $a$ in $A$ is  of order 8.  And indeed,  $b \in A$
interchanges the  two real embeddings so  that we end up  with exactly
one icosahedron  in this case.  One  easily checks that the  two other
automorphisms  $d$  and $ad$  have  trace  $-1$  resp.\ $1$  in  their
3-dimensional associated real representation,  so that we know already
where we shall encounter this icosahedron again. 
\end{proof}

\bn
In three of the nine cases of Table~\ref{cases} one can predict a
solution right away: If the representation $\delta$ of the  $C_2$ is
restriction of the natural matrix representation of $A$ connected to
the realization of  $A$ as geometric automorphism group of the regular
icosahedron. These are the cases\\  
$\delta (a)={\rm diag}(1,-1,-1), \delta (ad)={\rm diag}(-1,1,1),\delta
(d)={\rm diag}(-1,-1,-1)$.\\ 
We start with the third case because it relates to
Proposition~\ref{einlgeomd-1} of the introduction. 

\begin{lemma}\label{2,3}
In case $\delta (d) = {\rm diag}(-1,-1,-1)$ there are two formal {\sc
  Gram}-matrices corresponding to the two lines of the following table yielding four real  {\sc Gram}-matrices: 
\[
\begin{array}{|c|c|c|c|c|}
\hline
\mbox{Stabilizer in } A&  d_G,r_{1,G} & r_G & r_{f,G}&\mbox{cf. also La.}\\
 \hline
 C_2 \times A_5&2,\quad 2 &2&2 &\mbox{\ref{1,1}, \ref{3,2}}\\
 \hline 
 C_2 \times D_{10}&2,\quad 2 &2&2 &\mbox{\ref{1,1}, \ref{3,2}}\\
 \hline 
\end{array}
\]
Their field of definition $F_G$ is $\Q [\sqrt{5} ]$.\\
In both cases  $\delta (a) = {\rm diag}(1,-1,-1)$ and  $\delta (ad) = {\rm diag}(-1,1,1)$.
\end{lemma}

\begin{proof}
 1.) (Setting up the equations and solving them formally)
 Let $M\in \R^{3 \times 12}$ be the coordinate matrix. If the first
 column is zero, then so is the  twelfth, since $d$ interchanges 1 and
 12. But we are dealing with the case that all twelve vertices are
 pairwise different. Hence the first column is not zero. Because of
 the action of ${\rm O}_3(\R )$ we may assume  that it is of the form
 $(*,0,0)^{tr}$. There are two cases to be discussed: the special
 case, where  the first two columns are linearly dependent, and the
 generic case, where the first two columns are   linearly
 independent. In the first case the second column is of the same type
 as the first and their  equations, the intertwining condition
 (linear!) and the edge length condition (quadratic), quickly   lead
 to a contradiction. So we only have to deal with the generic
 case. Again because of the action   of the orthogonal group, we may
 assume that the second column is of the form  $(*,*,0)^{tr}$. Taking
 the action of $d$ into account, we obtain a system of 30 quadratic
 equations in 16 indeterminates.  The primary decomposition of the
 relation ideal yields the following list of residue class fields: 
 28 of degree 4, 20 of degree 8, 4 of degree 12, and 8 of degree 24. \\
 2.) (Analysis of the solutions) The 8 residue class fields of degree
 24 have no real embeddings  and must therefore be discarded. The same
 applies to the 4  residue class fields of degree 12.  Each of the 20
 residue class fields of degree 8 does have real embeddings, but none
 of them leads  to a positive semidefinite {\sc Gram}-matrix by the
 following argument: We look at the difference of the   3rd and 4th
 column of $M$. These two columns are the remaining vertices of the
 two triangles of   the icosahedron formed by the columns or rather
 vertices 1,2,3 and 1,2,4. Since the triangles are  equiangular of
 edge length 1 the squared distance  of 3 and 4 must lie between  0
 and 3. It so happens  that for all embeddings of the present cases
 this distance gets a negative value or a value bigger than 3, 
 which is impossible.\\
 Finally the 28 solutions of degree 4 all have real embeddings and
 lead to formal {\sc Gram}-matrices  over $\Q [x]/(x^2-5)$, which  for
 each of the two real embeddings yield positive semidefinite   {\sc
   Gram}-matrices. More precisely, the traces of the 24 formal {\sc
   Gram}-matrices have   minimal polynomial $\lambda^2-15\lambda+45$
 (four times) and $\lambda^2-15\lambda+269/5$ (24 times).   Writing
 them over the same field makes the 4 equal, i.e. gives one  {\sc
   Gram}-matrix over  $\Q [x]/(x^2-5)$, and the other 24 solutions
 yield 6  {\sc Gram}-matrices. The first one is fixed  by the
 operation of $A$, the other six form one orbit. In both cases the two
 real embeddings lead to positive semidefinite {\sc Gram}-matrices
 with distinct   real traces. 
\end{proof}

\begin{lemma}\label{1,1}
  In case $\delta (a) = {\rm diag}(1,-1,-1)$ there are
  $8$ formal {\sc Gram}-matrices and, up to equivalence, $19$
real  {\sc Gram}-matrices, with details listed in the following table:
\[
\begin{array}{|c|c|c|c|c|}
\hline
\mbox{Stabilizer in } A&  d_G, r_{1,G} & r_G & r_{f,G}&\mbox{cf. also La.}\\
 \hline
 C_2 \times A_5&2,\quad 2 &2&2  &\mbox{\ref{2,3}, \ref{3,2}}\\
 \hline 
 C_2 \times D_{10}&2,\quad 2 &2&2  &\mbox{\ref{2,3}, \ref{3,2}} \\
 \hline 
 C_2 ^2    & 2,\quad 2 & 2& 2&\mbox{\ref{3,2}}\\
 \hline
C_2 \times D_{10}&4,\quad 2 &2&2 &\mbox{\ref{2,2}, \ref{3,2}}\\
   \hline
 C_2 \times D_{6}&4,\quad 4 &2&2&\mbox{\ref{2,2}, \ref{3,2}} \\
   \hline  
{C_2}^2&24,\quad 10 &6&3&\mbox{\ref{3,2}} \\
   \hline 
{C_2}^2&30,\quad 18 &6&1& -\\
   \hline 
C_2&172,\quad 48 &20&5& -\\
   \hline 
\end{array}
\]
\end{lemma}

\begin{proof}
Since  this is  a  difficult  case, it  is  necessary  to prepare  the
computation of  the coordinate matrix  $M \in \R^{3 \times  12}$ well.
First of all  the first row lies in the  $6$-dimensional eigenspace of
$\pi (a)$ for  the eigenvalue 1, and  the second and third  row in the
$6$-dimensional  eigenspace for  the eigenvalue  $-1$. Hence  we start
with 18 indeterminates. The first column  of $M$ cannot be of the form
$(*,0,0)^{tr}$,  because otherwise  it would  be equal  to the  second
column   contradicting  our   assumption   of   having  12   different
vertices. By  applying a two  dimensional
rotation to modify  the last
two rows  without changing the  {\sc Gram}-matrix, we may  assume that
the  first  row  is  of  the form  $(*,*,0)^{tr}$.  Since  $\{1,  2\}$
represents an edge  and hence is of square length  1, we conclude that
the   first   column    can   be   chosen   to   be    of   the   form
$(*,1/2,0)^{tr}$. This  reduces the number  of variables from 3  to 1.
The difference of  columns 3 and 4 represents  the diagonal orthogonal
to  the edge  $\{  1, 2\}$,  proving  that  column 3  is  of the  form
$(*,0,*)^{tr}$ and  hence reducing the  number of variables by  1. The
reason why this  worked so well is  that the first two  cycles of $a$,
namely   $(1,2),  (3,4)$   represent  an   edge  and   its  orthogonal
diagonal.  The same  applies to  the last  two cycles  $(10, 12),  (9,
11)$. Since the  action of the two dimensional orthogonal  group is no
longer  available, we  get a  reduction by  1 only  for the  number of
variables. Finally, the center of mass condition gives us that the sum
of the  coefficients of the  first row is 0,  so that we  have another
reduction by one variable and end up with 13 variables.
\\
The rather time consuming computation of the associated prime ideals
with MAGMA leads to 72 maximal ideals with residue class fields of
degree 2 (48 times), 4  (10 times), 6 (4 times) , 8 (4 times),  16 (2
times), 48 (2 times), 120 (once), and 688 (once). We shall treat each
degree separately. 
\\
\underline{Degree 2}: Of the 48 residue class fields only 20 yield a
solution with 12 different  vertices. The traces of the 20 formal {\sc
  Gram}-matrices have three different minimal polynomials 
over the rationals with the following roots:  $15/2\pm (3/2)\sqrt{5}$ (4 times), 
$15/2\pm (7/10)\sqrt{5}$ (8 times), and $71/10\pm (5/6)\sqrt{5}$ (8
times). The formal {\sc Gram}-matrices with the  same minimal polynomial also have the same orbit lengths under the action of $A$, namely $1, 6, 30$.
Since the centralizer $C_A(a)=\langle a, b, d\rangle$ of $a$ in $A$
acts on the set of all formal $a$-invariant {\sc Gram}-matrices, its
orbits must be of length 1 in the first case, and of length 2 in the
second and third cases. Now writing the {\sc Gram}-matrices in terms
of polynomials of  their traces, we can compare them and find that our
list of 20 reduces to a set of 5 different {\sc Gram}-matrices, namely
one with automorphism group $A$, an orbit of 2 with automorphism group 
isomorphic to $C_2\times D_{10}$, and an orbit of 2 with automorphism
group ${C_2}^2$.  The first two {\sc Gram}-matrices have shown up
already in Lemma~\ref{2,3}, where their occurrence was already
predicted via the linear action of $d$ which is by multiplication by
$-1$,  as one easily checks.   Also  the linear action of the other
elements in the {\sc Sylow} 2-subgroups of the stabilizers  are easily
computed, cf. the comments on the automorphism groups, thus justifying
the entries in the last column of the table.
\\
\underline{Degree 4}:  Of the 10 residue class fields only 8 yield a solution with 12 different 
vertices. Similarly as in the degree-2-case one finds exactly two
non-isomorphic {\sc Gram}-matrices,
but unlike in the previous case none of them is positive semidefinite.
\\
\underline{Degree 6}: None of the four residue class fields has a real embedding. 
\\
\underline{Degree 8}: All four residue class fields lead to solutions
with 12 different vertices. All of them have the same minimal
polynomial for the trace and the orbits under $A$ are all of  length
6. The coefficients of the {\sc Gram}-matrices can be  represented as
polynomials in   the traces of the {\sc Gram}-matrices. Another
variable $co$, which can be interpreted as a cosine of a certain angle,
has minimal polynomial over the rationals of degree 2. This can be
used to  factor the minimal polynomial of the trace $t$ in two quadratic
factors and to see which of the  two can be discarded as discussed
below.    At this stage one sees that there are only two different
formal {\sc Gram}-matrices, which fall in one orbit under the
centralizer of $a$ in $A$. So we have to deal with only one formal
{\sc Gram}-matrix. The minimal polynomial of $co$ has two roots,
namely $-2\pm \sqrt{5}$. We conclude $co= -2+ \sqrt{5}$, because the
other root has absolute value bigger than 1. This makes the  minimal
polynomial of $t$ of degree 2 with solutions  
$\frac{9}{2}+{\frac {7}{10}}\,\sqrt {5}\pm \frac{2}{5}\,{5}^{3/4}$,
both of which give valid solutions (with different eigenvalues). A
{\sc Sylow} 2-subgroup of the automorphism group is given by $\langle
a, d\rangle$. Unlike in the other case with automorphism group 
$C_2 \times D_{10}$ the central element $d$ acts with eigenvalues
$1,1,-1$ on the 3-space.\\ 
\underline{Degree 16}:  Both residue class fields lead to solutions
with 12 different vertices. All of them have the same minimal
polynomial for the trace and the orbits under $A$ are all of  
length 10, proving that the two formal {\sc Gram}-matrices fall in one
orbit under the centralizer of $a$ in $A$, so that we only have to
look at one of them.  Again the  entries of {\sc Gram}-matrices can be
written as polynomials of the {\sc Gram} traces,  which generate a
totally real field extension of the rationals of degree 4. However
only two of the real embeddings lead to a positive semi-definite {\sc
  Gram}-matrix.  The {\sc Sylow} 2-subgroup of the automorphism group
is $\langle a, d \rangle$, however $d$ acts  not as scalar matrix on
the real 3-space but with eigenvalues $-1,1,1$. 
\\
\underline{Degree 48}: Both residue class fields lead to solutions
with 12 different vertices. All of them have the same minimal
polynomial for the trace and the orbits under $A$ are all of  length
30. The two minimal polynomials are equal and of degree 12. In both
cases, the $1,1$-entry of the {\sc Gram}-matrix is a primitive element
of the field $F_G$ generated by all entries. Again the minimal
polynomials over the rationals are equal and both of degree 24.
Writing the entries of both  {\sc Gram}-matrices  as
polynomials over these primitive elements makes them equal, so that we
have only one  formal  {\sc Gram}-matrix for the present case.  We
have 10 real embeddings of the field $F_G$, of which only 6 lead to a
positive semidefinite real  {\sc Gram}-matrix. The stabilizer of the
formal  {\sc Gram}-matrix in $A$ is  $\langle a, bd\rangle \cong
V_4$. Clearly, $a$ and $bd$ fix each one of the resulting real  {\sc
  Gram}-matrices, but $d$ permutes them in three 2-cycles. Note,
$N_A(\langle a, bd\rangle )/\langle a, bd\rangle$ is isomorphic to the
{\sc Galois} group of $F_G$ over the field generated by the trace of
the  {\sc Gram}-matrix, the latter acting equivalently on the set of
the 6 relevant real roots. So we end up with $3$ inequivalent
embeddings. The  eigenvalues of $bd$ and $abd$ in the action on real
3-space can be read off from
$a \mapsto (1,-1,-1), bd \mapsto (1,1,-1), abd \mapsto  (1,-1,1)$. 
\\
\underline{Degree 120}: The residue class field leads to one formal
solution with 12 different vertices and the  orbit of the formal {\sc
  Gram}-matrices under $A$ is of length 30 with stabilizer $\langle
a,b \rangle \cong V_4$.  The $1,7$-entry of the  {\sc Gram}-matrix
already generates the field $F_G$ generated by all entries, which
turns out to be of degree 30 over the rationals. This field has 18
real embeddings of which only 6 lead to positive semidefinite real 
{\sc Gram}-matrices. The trace $t$ of the formal {\sc Gram}-matrix
generates a field extension of degree 5, so that the field extension
$(F_G/ \Q [t])$ is of degree 6. It turns out to be a {\sc Galois}
extension  with {\sc Galois} group isomorphic to  $N_A(\langle a,
b\rangle )/\langle a, b\rangle \cong C_6$. Indeed, we end up with one
real {\sc Gram}-matrix up to equivalence. On 3-space all three
elements $a, b, ab$ of the  stabilizer act linearly with trace $-1$.
\\
\underline{Degree 688}: 
The residue class field leads to one formal solution with 12 different
vertices. The  orbit of the formal {\sc Gram}-matrices under $A$ is of
length 60 with stabilizer $\langle a\rangle  
\cong C_2$.  The sum of the $1,1$-entry and the $5,5$-entry of the 
{\sc Gram}-matrix already generates the field $F_G$ generated by all
entries. It has degree 172 and has 48 real
embeddings of which only 20 lead to positive semidefinite real 
{\sc Gram}-matrices. The trace $t$ of the formal {\sc Gram}-matrix
generates a field extension of 
degree 43, so that the field extension  $(F_G/ \Q [t])$ is of degree 4. It turns out to be
a {\sc Galois} extension  with {\sc Galois} group isomorphic to 
$N_A(\langle a\rangle )/\langle a\rangle \cong V_4$. Indeed,
the centralizer of $a$ in $A$ is $\langle a,b,d\rangle \cong C_2^3$ and subdivides
the twenty {\sc Gram}-matrices into 5 orbits of length $4=|V_4|$.
\end{proof}

\bn  Whenever  in   the  proofs  of  the  subsequent   lemmas  a  {\sc
  Gram}-matrix shows up  which is invariant under $a$, we  know by the
preceding lemmas that  this {\sc Gram}-matrix occurred  already in one
of the Lemmas~\ref{1,1}, \ref{1,3} and its isomorphism type can easily
be read off  from the various invariants used. On  the other hand, the
previous lemmas also  tell us, in which of the  subsequent lemmas such
{\sc Gram}-matrices will occur.

\begin{lemma}\label{3,2}
  In case $\delta (ad) = {\rm diag}(-1,1,1)$ there are $12$ formal
  {\sc Gram}-matrices and, up to equivalence, $29$
  real  {\sc Gram}-matrices, with details listed in the following
  table\footnote{The subdivision of the table reflects the two cases in the proof.}:
\[
\begin{array}{|c|c|c|c|c|}
\hline
\mbox{Stabilizer in } A&  d_G, r_{1,G} & r_G & r_{f,G}&\mbox{cf. also La.}\\
 \hline
 C_2 \times A_5&2,\quad 2 &2&2  &\mbox{\ref{2,3}, \ref{1,1}}\\
 \hline 
 C_2 \times D_{10}&2,\quad 2 &2&2  &\mbox{\ref{2,3}, \ref{1,1}}\\
 \hline 
 {C_2}^2   & 2,\quad 2 & 2& 2 &\mbox{\ref{1,1}}\\
   \hline 
D_{10}&2,\quad 2 &2&2 &- \\
\hline
D_{6}&2,\quad 2 &2&2 &- \\
  \hline
{C_2}^2&24,\quad 10 &6&3 &\mbox{\ref{1,1}} \\
 \hline
 C_2& 36,\quad 12 &8&4 & -\\
 \hline \hline
C_2&168,\quad 40&24&6 &-\\
\hline
C_2 \times D_{6}&4,\quad 4 &2&2 &\mbox{\ref{1,1}, \ref{2,2}} \\
   \hline 
   C_2 \times D_{10}&4,\quad 2 &2&2 &\mbox{\ref{1,1}, \ref{2,2}} \\
   \hline
 {C_2}^2&8,\quad 4 &2&1 & \mbox{\ref{2,1}, \ref{1,3}}\\
   \hline  
  D_{10}&4,\quad 2&2&1 & -\\ \hline 
\end{array}
\]
\end{lemma}

\begin{proof}
The coordinate matrix has its first row in the (row) eigenspace of the
permutation matrix for $ad$ for the eigenvalue $-1$, which is of
dimension~4, the second and third row lie in the  eigenspace for the
eigenvalue $1$, which is of dimension~8. Hence we have 20 variables
for  the coordinate matrix. The fourth column is of the form
$(0,*,*)^{tr}$ and can be  assumed to be of the form
$(0,*,0)^{tr}$. Because the cycle $(3,11)$ of $ad$ consists of the
vertices of an edge of the icosahedron, the first entry $x_3$ of the
third column satisfies $x_3^2=1/4$. W.l.o.g. we may assume
$x_3=1/2$. (This makes the first entry of the eleventh column $-1/2$.)
Also $(4,9)$ is a cycle of $ad$ and  represents an edge of the
icosahedron at the same time. Since we are no longer allowed  to
multiply the first row of the coordinate matrix by $-1$, we now have
two possibilities  for the first entry $x_4$ of the forth column,
namely case 1:  $x_4=1/2$ and case 2:  $x_4=-1/2$.  In both cases the
number of remaining variables is 18. This number can be reduced to 16,
since we have the center of mass condition by which the sum of the
entries in the second and third row is zero, which are two new
equations unlike to the first row. From here on we treat the two cases
separately.\\ 
\underline{Case 1}: We get 63 prime ideals in the primary
decomposition.  All except for 4 prime ideals lead to icosahedra with
12 different vertices. The residue  class fields in these 59 relevant
cases have  degree 
2 (20 times), 4 (28 times), 6 (4 times), 8 (2 times), 48 (2 times), 96 (once), 144 (2 times).\\
\underline{Degree 2, case 1} All {\sc Gram}-matrices which come up are
also fixed by $a$ or a  conjugate of $a$ such as $b$ or $ab$. Hence
they are known already. To be more precise,  if one writes the {\sc
  Gram}-matrices in terms of their traces, one ends  up with 5
different formal {\sc Gram}-matrices coming in three orbits under the
centralizer of $ad$ in $A$. Since we have treated already all
$a$-invariant {\sc Gram}-matrices, these orbits must  correspond to
the first three  {\sc Gram}-matrices in Lemma~\ref{1,1}.\\ 
\underline{Degree 4, case 1} In all 28 cases the entries of the formal
{\sc Gram}-matrices can  be written in terms of the traces
of their  {\sc Gram}-matrices so that equivalence becomes easy  to
check. The rational minimal polynomials of the traces are all of
degree 2. There are 8, 8, 8, resp.\ 4 cases where the $A$-orbit length
of the {\sc Gram}-matrices is 30, 10,  12, resp.\ 20. The 8 of orbit
length 30 yield two $C_A(ad)$-orbits of length 2, one with  stabilizer
$\langle ad,b\rangle$, hence again isomorphic to the third in
Lemma~\ref{1,1}, the  second orbit yielding only indefinite  {\sc
  Gram}-matrices. The 8 {\sc Gram}-matrices with $A$-orbit length 10
boil down to an orbit of two, which, however, are indefinite. The 8 of
orbit length 12 give 4 different  {\sc Gram}-matrices, which form one
orbit under  $C_A(ad)$. Both real  embeddings are positive
semidefinite, thus giving a new formal {\sc Gram}-matrix whose
stabilizer is isomorphic to $D_{10}$. Finally, the 8 of orbit length
20 give 4 different  {\sc Gram}-matrices, which form one orbit under
$C_A(ad)$. Both real  embeddings are positive semidefinite, thus
giving a new formal {\sc Gram}-matrix whose  stabilizer is isomorphic to $D_{6}$.\\
\underline{Degree 6, case 1} There are no real solutions. More
precisely the traces of the two  {\sc Gram}-matrices have a rational
minimal polynomial of degree 2 with negative discriminant.\\ 
\underline{Degree 8, case 1} This case yields two real {\sc
  Gram}-matrices fixed under $ab$. Both are  not positive semi-definite
and hence  ruled out.\\ 
\underline{Degree 48, case 1} The {\sc Gram}-matrices are fixed under
$b$, which is conjugate to $a$. Hence, this case was treated
earlier. Indeed, the two  {\sc Gram}-matrices are equivalent to the
ones  obtained in Lemma~\ref{1,1} with automorphism group ${C_2}^2$
and $d_G=24$.\\ 
\underline{Degree 96, case 1}  Here no real solution exists, because
the rational minimum  polynomial of the trace of the  {\sc
  Gram}-matrix is of degree 6 without real roots.\\ 
\underline{Degree 144, case 1} There are two {\sc Gram}-matrices over
a field of degree 36 over the  rationals. The stabilizer in $A$ is
generated by $ad$ in both cases. Since the normalizer of  $\langle
ad\rangle $ is of order 8, one of these two solutions gives us already
an orbit of four  {\sc Gram}-matrices. This can only be the case if a
subgroup of order 2 of  $N_A(\langle ad\rangle )/\langle ad\rangle$ is
induced by field automorphisms. In any case, it suffices to look at
the first formal {\sc Gram}-matrix. Its $(7,1)$-entry can be chosen as
primitive element of the field. It allows us to look at the real
completions, of which there  are 12 with 8 of them yielding a positive
definite {\sc Gram}-matrix. It turns out that  $b \in A$ induces the
expected field automorphism because it permutes the 8 real places in 
4 transpositions, whereas $d$ for instance does not. (Ordering the 8
real roots according to their increasing size, $b$ yields the
permutation $(1,2)(3,6)(4,5)(7,8)$.) Hence we end  up with 4 real {\sc
  Gram}-matrices up to equivalence.\\   

\noindent
\underline{Case 2}: Using the MAGMA {\sc Groebner}-basis routine and
the Involutive basis routine we see that the residue class algebra of
our equations is of $\Q$-dimension 928. It seems to be very hard  to
obtain automatically the associated prime ideals. Therefore we compute
the minimal polynomial  of the trace of the {\sc Gram}-matrix, which
turns out to be of degree 70 only. It factors into irreducible factors
of degrees 2, 12, 42, 4, 4, 4, 2. The first factor has no real roots
and has therefore to be discarded.  The remaining ones have 6, 12, 4,
2, 2, resp.\ 2 real roots. We discuss each factor separately. Note, at
this stage we do not yet know whether each factor leads to a maximal
ideal.\\
\underline{Degree-12-factor case 2} The residue class algebra is of
$\Q$-dimension 96, leads to a {\sc Gram}-matrix over a field of degree
24  with stabilizer $\langle ad, ab \rangle\cong {C_2}^2$ under the
action of $A$. Clearly, this was treated earlier in
Lemma~\ref{1,1}. Note this form also came up in Case 1, degree 48
above in this proof.\\ 
\underline{Degree-42-factor case 2} The residue class algebra is of
$\Q$-dimension 672, leads to a  {\sc Gram}-matrix over a field $F_Q$
of degree 168 as follows: The sum of the first and third  diagonal
element of the {\sc Gram}-matrix turns out to have an irreducible
minimal polynomial of degree $168=672/4$ and all the other entries of
the {\sc Gram}-matrix are polynomials in this element. The stabilizer
of the {\sc Gram}-matrix is $\langle ad\rangle\cong C_2$ under the
action of $A$. The field $F_Q$ has 40 real embeddings, 24 of which
lead to a  positive semidefinite real {\sc Gram}-matrix. Since there
is only one formal {\sc Gram}-matrix in this case, the factor group
$N_A(\langle ad\rangle )/\langle ad\rangle \cong V_4$ can be embedded
into the automorphism group of $F_Q$ over the rationals and
distributes the relevant real  embeddings into orbits of isometric
real {\sc Gram}-matrices, namely 6 orbits of 4 embeddings
each. Indeed, these real matrices have 6 different traces, we end up
with 6 classes of real {\sc Gram}-matrices in this case. We note that
this case was computationally hard.\\ 
\underline{First degree-4-factor case 2} The residue class algebra is
of $\Q$-dimension 32 and decomposes into a direct sum of two fields
of degree 16 over the rationals. Both fields  lead to a  {\sc
  Gram}-matrix over a field $F_Q$ totally real of degree 4 and have
$C_2 \times D_6$ as stabilizers in $A$. Of the four real  embeddings
only two lead to positive semidefinite {\sc Gram}-matrices. The  two
specializations of one formal {\sc Gram}-matrix are
inequivalent. Therefore, it is clear that  the two formal {\sc
  Gram}-matrices form an orbit under the normalizer of the stabilizer
$\langle ad, d\rangle$ in $A$. In particular only one  {\sc
  Gram}-matrix has to be considered.  The traces of the linear actions
on 3-space are $-1$, $1$, $1$, so that this form should come up in
Lemmas~\ref{1,1} and \ref{2,1}.  
\\
\underline{Second degree-4-factor case 2} The residue class algebra is
of $\Q$-dimension 32 and  decomposes into a direct sum of four fields
of degree 8 over the rationals. Each field leads to a {\sc
  Gram}-matrix over a field $F_Q\cong \Q [5^{1/4}]$ of degree 4 and
have $C_2 \times D_{10}$ as stabilizers in  $A$, all of them
containing $a$ and therefore have come up earlier (in the case
$a\mapsto (1,-1,-1)$). From the earlier case all the formal
{\sc Gram}-matrices must be equivalent. \\ 
\underline{Third degree-4-factor case 2} The residue class algebra is
of $\Q$-dimension 32. It is already a field and  leads to a  {\sc
  Gram}-matrix over a field $F_Q$ of degree 8 and has $\langle a,
d\rangle \cong {C_2}^2$ as stabilizer in $A$. Therefore, it has come
up earlier in the case $a \mapsto (-1,+1,+1)$, cf. Lemma~\ref{1,3}.\\ 
\underline{Degree-2-factor case 2} The residue class algebra leads to
four residue class fields of degree 4 over the rationals. The four
resulting {\sc Gram}-matrices all have stabilizer isomorphic to
$D_{10}$ and form four different fixed points for $ad$ in accordance
with the table of marks of $A$. They fall  into two orbits under the
action of $\langle a\rangle$, but lie in one orbit under $A$. The
field $F_Q$, which is isomorphic to $\Q[5^{1/4}]$, has two real
embeddings. These yield isometric icosahedra, the normalizer of the
stabilizer modulo the  stabilizer acts on the field interchanging the
two real embeddings.  
\end{proof}

\begin{lemma}\label{2,2}
  In case $\delta (d) = {\rm diag}(-1,1,1)$ there are two formal {\sc
    Gram}-matrices and four real  {\sc Gram}-matrices, with details
  listed in the following table: 
\[
\begin{array}{|c|c|c|c|c|}
\hline
\mbox{Stabilizer in } A&  d_G, r_{1,G} & r_G & r_{f,G}& \mbox{ cf. also La.}\\
 \hline
C_2 \times D_{10}&4,\quad 2 &2&2&\mbox{\ref{1,1}, \ref{3,2} }\\
   \hline
 C_2 \times D_{6}&4,\quad 4 &2&2&\mbox{\ref{1,1}, \ref{3,2} }\\
   \hline  
\end{array}
\]
\end{lemma}

\begin{proof}
The coordinate matrix has its first row in the (row) eigenspace of the
permutation  matrix for  $d$  for  the eigenvalue  $-1$,  which is  of
dimension 6,  the second and third  row lie in the  eigenspace for the
eigenvalue  $1$, which  is  also of  dimension 6.  Hence,  we have  18
variables  for the  coordinate matrix.  The center  of mass  condition
reduces this number by 2. As usual  the first column can be assumed to
be of the form $(*,*,0)^{tr}$.  However, this leads to an infinite number
of solutions for  the coordinate matrices. The  computation shows that
the  infinity occurs  already  if  the first  column  is  of the  form
$(*,0,0)^{tr}$. This is  a case where the action  of the 2-dimensional
orthogonal group is still in operation, allowing us to assume that the
second column  is of  the form  $(*,*,0)$, which  then results  into 4
solutions only and  just one {\sc Gram}-matrix. As for  the other case
where the second entry of the  first column is non-zero, one gets also
only  finitely  many  {\sc  Gram}-matrices.  All  of  them  fall  into
$A$-orbits of lengths 6 and 10. This shows that the automorphism group
must  contain  an  element  conjugate  to $a$  under  $A$.  Since  all
possibilities for linear  actions of $a$ have  already been discussed,
we can look up there, which cases for $a$ contain $d$ in its presently
discussed linear action.  Note, since $d$  is central in $A$, the full
$A$-orbits show  up in  our computation. This  explains why  the issue
with the infinite number of  solutions for the coordinate matrices did
not show up at the two relevant places.  
\end{proof}

\begin{lemma}\label{2,1}
 In case $\delta (d) = {\rm diag}(1,-1,-1)$ the equational ideal $I$ of Definition~\ref{formalGram} is of dimension 1.
\end{lemma}

\begin{proof}
The usual ansatz with one entry in the third row set zero to make the 
equations for the coordinate matrix rigid leads to an ideal of which a
{\sc Groebner}-basis can be computed by MAGMA. This can be turned into 
a {\sc Janet}-basis yielding the following {\sc Hilbert} series:
\begin{multline*} 
1+16t+121t^2+576t^3+1625t^4+1987t^5+1540t^6+1371t^7+1323t^8+\\
1320t^9+\frac{1320t^{10}}{1-t}
\end{multline*}
telling us that the dimension is 1.  We could not prove that the ideal
was  prime.  As   was  to  be  expected,  the   {\sc  Gram}-matrix  of
Lemma~\ref{1,3}  corresponds to  a maximal  ideal containing  $I$. The
assumption that two  vertices coincide leads in each  case to finitely
many maximal ideals containing $I$  so that generically we have twelve
vertices. However,  it is  not clear  at this  stage whether  we have
infinitely many  real solutions leading to  positive semidefinite {\sc
  Gram}-matrices. This will be proved in the next section.  
\end{proof}

\bn
All these lemmas of this section taken together prove Theorem~\ref{Hauptsatz}.

\section{A curve of icosahedra}\label{CurvIco}
In this section we prove the remaining part of Theorem~\ref{Uebersicht},
namely that there are infinitely many $\langle d\rangle$-invariant 
icosahedra. The previous computations allow us to define a vector field,
at least one of its integral curves consists of icosahedra:

\begin{proposition}\label{curve}
There exist $\varepsilon > 0$
and a non-constant real analytic map $\Phi: [0, \varepsilon) \to \R^{3 \times 12}$ 
such that $\Phi(t)$ is the coordinate matrix of
a $\langle d \rangle$-invariant icosahedron
for all but finitely many $t \in [0, \varepsilon)$.
In particular, there exist infinitely many isometry types
of $\langle d \rangle$-invariant icosahedra.
\end{proposition}

\begin{proof}
Let
\begin{equation}\label{CurveEquation}
p_1 = 0, \quad
p_2 = 0, \quad
\ldots, \quad
p_{15} = 0
\end{equation}
be the (linearly independent) quadratic equations in the unknown
entries $y_1$, $y_2$, \ldots, $y_{16}$ of the coordinate matrix $M$
taking into account the $\langle d \rangle$-symmetry (cf.\ also Section~\ref{sec:formalGramMatrices} as well as the comments about
factoring out the action of the 3-dimensional orthogonal group at the
beginning of Section~\ref{Classify} making the fibres of $M\mapsto M^{tr}M$
finite).
Let
${\rm D}p$ be the Jacobian matrix of $p = (p_1, p_2, \ldots, p_{15})$.
Due to the {\sc Laplace}-expansion for determinants, the vector
\[
\tau(y_1, \ldots, y_{16}) :=
\left(
\det {\rm D}p|_1, \,
\ldots, \,
(-1)^{i-1} \det {\rm D}p|_i, \,
\ldots, \,
-\det {\rm D}p|_{16}
\right)
\]
satisfies
\[
\tau(y_1, \ldots, y_{16}) \, ({\rm D}p)^{tr} = 0,
\]
where ${\rm D}p|_i$ is the square submatrix of ${\rm D}p$ that is obtained by omitting the $i$-th column.
Hence, for each real solution $y^0 = (y^0_1, \ldots, y^0_{16})$ of
(\ref{CurveEquation})
the evaluation of the vector $\tau$ 
at $y^0$
is tangent to the algebraic curve defined by (\ref{CurveEquation})
at the point $y^0$.
Given any such solution $y^0$ such that $\tau(y^0)$ is non-zero, we consider
the following initial value problem
for $\phi(t) = (\phi_1(t), \ldots, \phi_{16}(t))$
on some interval containing $0$:
\[
\left\{
\begin{array}{rcl}
\phi'(t) & = & \tau(\phi_1(t), \ldots, \phi_{16}(t)),\\[0.2em]
\phi(0) & = & y^0.
\end{array}
\right.
\]
By standard theorems on ordinary differential equations,
there exist $\varepsilon > 0$ and
a real analytic map $\phi: [0, \varepsilon) \to \R^{1 \times 16}$
satisfying the above initial value problem.
By substituting $\phi_i$ for $y_i$ in $M$ we obtain a candidate for
a real analytic
map $\Phi: [0, \varepsilon) \to \R^{3 \times 12}$ as required.
Note that it still needs to be checked whether the columns of $\Phi(t)$
are pairwise distinct and whether $\Phi(t)^{tr} \Phi(t)$ is a
positive semi-definite {\sc Gram}-matrix of rank 3.
As for the first condition, {\sc Groebner} basis computations (cf. also proof 
of Lemma~\ref{2,1}) show
that there are only finitely
many (real) solutions of (\ref{CurveEquation}) such that two columns of
the corresponding coordinate matrices are equal.
(Hence, even if $y^0$ is chosen as a solution of (\ref{CurveEquation})
defining a degenerate icosahedron with two coincident vertices, the
resulting $\Phi(t)$ will define an icosahedron with pairwise distinct
vertices for sufficiently small $t > 0$.)
Moreover, for sufficiently small $t > 0$
the conditions on $\Phi(t)^{tr} \Phi(t)$ follow from
the corresponding ones satisfied by $\Phi(0)^{tr} \Phi(0)$ by continuity.
Hence, by choosing a smaller $\varepsilon$ if necessary, a real analytic
map $\Phi$ as required is obtained.
Finally, since the map $t \mapsto \Phi(t)^{tr} \Phi(t)$ is not constant
and since only finitely many isometry types of icosahedra correspond to
a {\sc Gram}-matrix $\Phi(t)^{tr} \Phi(t)$, we conclude that $\Phi$ defines
infinitely many pairwise inequivalent $\langle d \rangle$-invariant
icosahedra.
\end{proof}

We used the numeric ODE solver in Maple~2017 to find an
approximate solution $\Phi$ starting with a $\langle d \rangle$-invariant
icosahedron given in terms of algebraic numbers.
The computation was carried out with a precision of $20$ digits,
resulting in approximate icosahedra with a residual error of at most
$0.14 \times 10^{-8}$ for $1000$ time steps in the interval $[0, 0.00018]$.

\section{Some geometric and combinatoric invariants}
Beyond the automorphism group of an icosahedron we propose some geometric 
invariants which might give some idea of how the triangles of the icosahedron
are arranged in 3-space. \\

\begin{definition}
1.) Let $\Delta_1, \Delta_2$ be two equiangular triangles in {\sc Euclidean} 3-space. A point $P$ in  the affine space spanned by the vertices of $\Delta_1$
and $\Delta_2$ is called \underline{central} for 
$\Delta_1, \Delta_2$ if one of the following three equivalent conditions is
satisfied:\\
i.) (Circumsphere) There is a sphere with midpoint $P$ passing through
all vertices of  $\Delta_1$ and $\Delta_2$.\\
ii.) (Insphere) There is a sphere with midpoint $P$ passing through
all the midpoints of the edges of $\Delta_1$ and $\Delta_2$.\\
iii.) (Centersphere) There is a sphere with midpoint $P$ intersecting
the convex hulls of $\Delta_1$ and $\Delta_2$ tangentially in their
centers  of their incircles. \\ 
2.) Let $X$ be an icosahedron in 3-space. A point $P$ in 3-space is called 
\underline{significant} of strength $k\geq 2$ for $X$, if there are
$k$ triangles of $X$ such that $P$ is central for any pair of them
with the same insphere.  
\end{definition}

\bn  If two  equiangular  triangles with  different  midpoints have  a
central point, it is  unique: In case the two triangles  do not lie in
parallel  planes, it  is  the intersection  of  the orthogonal  middle
lines; in  case they lie in  different parallel planes, it  is the
midpoint of  the centers of  their incircles. Note, if  the orthogonal
middle lines  of the  triangles intersect,  the point  of intersection
need not be a central point  of the two triangles.  If two equiangular
triangles share exactly  one edge, they are either coplanar  or have a
central point.  The central points of  an icosahedron can be viewed as
a  substitute of  the  midpoint  of the  regular  icosahedron, as  the
example below will demonstrate.  Since  the defining equations for the
significant points  (together with the generators of the maximal ideal
defining a  formal {\sc Gram}-matrix)  are all over the  rationals, we
have an obvious remark:

\begin{remark}
1.) If two icosahedra belong to the same formal {\sc Gram}-matrix, they have 
the same number of $k$-significant points for any $k$.\\
2.) Let $s$ be the sum of the strengths of all significant points of an icosahedron.
Then $20 \leq s$ with equality if and only if the spheres of the significant points
form a partition of the faces of the icosahedron.\\
3.) There is a map from the set of bend edges of the icosahedron to its set of spheres or
equivalently to its set of significant points, which takes the edge to the intersection of
the two lines orthogonal to the faces passing through their centers.
\end{remark}

\begin{example}
1.) The regular icosahedron and its {\sc Galois} conjugate have exactly one
significant point. It is of strength 20.\\
2.) The icosahedra with symmetry group $C_2 \times D_{10}$ have three 
significant points. These have strengths $5,5,10$. The first two fall in 
one orbit under the symmetry group. The sets of triangles with the same
central point form a partition of the set of faces of $X.$ \\
3.) Two neighbouring triangles in an icosahedron are either coplanar or
give rise to a significant point of strength at least 2. This is the simplest
way in which significant points arise. Significant points of bigger 
strength are clearly more interesting. We call a significant point trivial, 
if it is of strength 2 and its two associated triangles share an edge. \\
4.) There is a unique icosahedron with symmetry group $D_{10}$ and field 
degree $d_G=4$. Here one computes easily that there is one significant point of
strength 10, two of strength 5,  and ten trivial ones. The five
triangles associated to a significant point of strength 5 share a
common vertex. These two vertices  
have combinatorial distance 3, i.e. form a cycle of $d$, the ten
remaining triangles belong to the strength-10-point. The 10 trivial
significant points  
bind together one triangle of the 10-belt with one triangle of the two 
5-caps so that in the end each triangle belongs to exactly two significant points.\\
5.) There is a unique icosahedron with symmetry group ${C_2}^2$ and field 
degree $d_G=30$. In this case all significant points are trivial: They are in
obvious bijection with the edges of the icosahedron.\\
6.) If a face has two equal face angles $\not= \pi$ to its neighbours then it gives rise to  a
significant point of strength at least 3. In case it is not bigger we call it a
3-trivial significant point. If it has all three face angles to its neighbours equal and $\not= \pi$,
it gives rise to a significant point of strength at least 4. In case it is not bigger we call it a
4-trivial significant point.
\end{example}

\noindent
{\bf Acknowledgement}:
 The calculations were partly done with Maple\footnote{Maple is a trademark of Waterloo Maple Inc.}. We thank Bernd Schulze for helpful comments.


\end{document}